\documentclass[11pt,reqno]{amsart}
\usepackage{diagrams}
\usepackage{amssymb}
\usepackage{cite}
\usepackage[pagewise]{lineno}
\numberwithin{equation}{section}

\theoremstyle{definition}
\newtheorem{thm}{Theorem}[section]

\newtheorem{prop}[thm]{Proposition}
\newtheorem{defi}[thm]{Definition}
\newtheorem{rem}[thm]{Remark}
\newtheorem{note}[thm]{Notation}

\DeclareMathOperator{\Ima}{\mathrm{Im}}

\DeclareMathOperator{\p3}{\mathbb{P}^3}

\DeclareMathOperator{\pr}{\mathrm{pr}}

\DeclareMathOperator{\mo}{\mathcal{O}}

\newcommand{\mr}[1]{\mathrm{#1}}
\newcommand{\mb}[1]{\mathbb{#1}}

\newcommand{\mc}[1]{\mathcal{#1}}
\newcommand{\ov}[1]{\overline{#1}}

\newcommand{\mf}[1]{\mathfrak{#1}}

\begin{document}

\title[Generators of the cohomology ring]
{Generators of the cohomology ring, after Newstead}

\author[S. Basu]{Suratno Basu}

\address{SRM University, Amaravathy, Mangalagiri Mandal
Guntur District,  Andhra Pradesh 522240, India}

\email{suratno.b@srmap.edu.in}

\author[A. Dan]{Ananyo Dan}

\address{School of Mathematics and Statistics, University of Sheffield, Hicks building, Hounsfield Road, S3 7RH, UK}

\email{a.dan@sheffield.ac.uk}

\author[I. Kaur]{Inder Kaur}

\address{Institut für Mathematik,
Goethe-Universität Frankfurt,
Robert-Mayer-Str. 6-8,
60325 Frankfurt am Main, Germany}

\email{kaur@math.uni-frankfurt.de}

\subjclass[2010]{$32$G$20$,  $32$S$35$, $14$D$07$,  $14$D$22$, $14$D$20$,  $14$H$60$, $55$R$40$}


\date{\today}

\begin{abstract}
The moduli space of rank $2$ stable, locally-free sheaves 
 with fixed odd degree determinant over a smooth, projective curve
 is a classical object. 
In the early 1970s, 
Newstead gave the generators of the cohomology ring of this moduli space.
There are generalizations of this result to higher rank, but 
nothing is known for the case when the underlying curve is singular.
In this article, we generalize Newstead's result to the case when the 
 underlying curve is irreducible, nodal. We show that the generators of the cohomology ring in the nodal curve case 
 arise naturally as degeneration of Newstead's generators in the smooth curve case.
\end{abstract}

\maketitle

\section{Introduction}
Throughout the article, the underlying field will be $\mb{C}$.
Let $C$ be a smooth, projective curve of genus $g \geq 2$, $d$ an odd integer 
and $\mc{L}$ an invertible sheaf of degree $d$ over $C$.
Denote by $M_C(2,d)$ the moduli space of stable locally-free sheaves of 
rank $2$ and degree $d$ over $C$, by 
$M_C(2,\mc{L})$ the  sub-moduli space of $M_C(2,d)$ parameterizing locally-free sheaves with determinant $\mc{L}$. We know that the rational cohomology ring 
of the moduli space $M_C(2,\mc{L})$ is finitely generated. So, it is 
natural to ask if the generators can be described explicitly.
This was answered by Newstead \cite{new1} in $1972$
and led to several subsequent results (see Atiyah-Bott \cite{atyang}, 
Biswas-Raghavendra  \cite{BR}, Kirwan \cite{kirw}).
 Knowing the generators of the cohomology ring
of the moduli space $M_C(2,\mc{L})$ has seen applications 
ranging from the computation of the Hodge-Poincar\'{e} polynomial to mirror symmetry
(see \cite{HT, HT2, HRV}).
Since moduli spaces of semi-stable sheaves with fixed determinant
over irreducible nodal curves exist, it is natural to ask 
if we can describe the generators of the cohomology ring of the moduli space
in this case.

Let $X_0$ be an irreducible, nodal curve with exactly one node, of genus $g \ge 2$, $\Delta$ the open disc
and $\pi:\mc{X} \to \Delta$ be a flat family of smooth, projective curves of genus $g$
degenerating to $X_0$ i.e., $\pi$ is smooth over $\Delta^*$ and $\pi^{-1}(0)=X_0$.
Let $\mc{L}$ be a relative odd degree invertible sheaf on $\mc{X}$. 
It is well-known that there are two different relative moduli spaces of semi-stable sheaves of rank $2$ and 
determinant $\mc{L}$. Denote by $U(2,\mc{L})$ the relative moduli space of stable, 
rank $2$ torsion-free
sheaves with determinant $\mc{L}$ defined over $\mc{X}$ 
(see \cite{bht3, sun1} for precise definition and construction) and 
 by $\mc{G}(2,\mc{L})$ the relative Gieseker moduli space of rank $2$, semi-stable sheaves defined over families of curves semi-stably equivalent to $\mc{X}$
(see \S \ref{sec:gies}). 
Both the moduli spaces $U(2,\mc{L})$ and $\mc{G}(2,\mc{L})$ agree over the punctured 
disc and have been studied in 
various contexts (see \cite{sun1, DK2, indpre, mumf}).
Denote by $\mc{G}_{X_0}(2,\mc{L}_0)$ (resp. $U_{X_0}(2,\mc{L}_0)$)
the central fiber of the relative moduli space $\mc{G}(2,\mc{L})$ (resp. 
$U(2,\mc{L})$).
 It is well-known that there is a proper morphism 
from $\mc{G}_{X_0}(2,\mc{L}_0)$ to $U_{X_0}(2,\mc{L}_0)$, which is not an isomorphism.
In \cite{K4},  Kaur showed that both the moduli spaces
$\mc{G}_{X_0}(2,\mc{L}_0)$ and $U_{X_0}(2,\mc{L}_0)$ are non-empty for any $\mc{L}_0$.
In \cite{mumf} the last two 
authors gave the generators for the rational cohomology ring of $U_{X_0}(2, \mc{L}_0)$,
in the  context of studying a generalized version of a conjecture of Mumford, which states that 
a certain set of relations between the generators of the cohomology ring is complete.
 The goal of this article is to describe the 
   generators of the rational cohomology ring of $\mc{G}_{X_0}(2,\mc{L}_0)$.
   We show that in fact, the rational cohomology ring 
   of $\mc{G}_{X_0}(2,\mc{L}_0)$ has twice as many 
generators as the rational cohomology ring of $U_{X_0}(2, \mc{L}_0)$
(compare with \cite[Theorem $1.1$]{mumf}). 
   Moreover, the proof given here is simpler and does not need much of the 
   technicalities used in \cite{mumf}. Finally, we show that these
   generators arise naturally as degeneration of Newstead's generators in the smooth curve case.

   The main result of this article is: 
   \begin{thm}\label{th:gen05}
  There exist distinct $\xi_j^{(i)} \in H^i(\mc{G}_{X_0}(2,\mc{L}_0),\mb{Q})$ for the following values of $(i,j)$ forming a minimal set of generators of 
    the cohomology ring $H^*(\mc{G}_{X_0}(2,\mc{L}_0),\mb{Q})$:
    
\begin{center}
    \begin{tabular}{| l | l | }
    \hline
    i & j \\ \hline
    2 & \{1,2\} \\ \hline
    3 & \{1, 2, ..., 2g-1\} \\ \hline
    4 & \{1, 2, 3\} \\    \hline
    5 & \{1, 2, ..., 2g-2\} \\ \hline
    6 & \{1, 2\}\\ \hline
    \end{tabular}
\end{center}
        \end{thm}

The idea is to use the theory of variation of mixed Hodge structures by Steenbrink \cite{ste1} and Schmid \cite{schvar}
to relate the mixed Hodge structure on the central fiber of the 
relative moduli space to the limit mixed Hodge structure on the generic fiber.
Denote by $\mc{G}(2,\mc{L})_\infty$ the generic fiber of the 
relative moduli space. By Steenbrink \cite{ste1}, $H^i(\mc{G}(2,\mc{L})_\infty, \mb{Q})$
is equipped with a (limit) mixed Hodge structure such that
the specialization morphism \[\mr{sp}_i: H^i(\mc{G}_{X_0}(2,\mc{L}_0), \mb{Q}) \to H^i(\mc{G}(2,\mc{L})_\infty, \mb{Q})\]
is a morphism of mixed Hodge structures for all $i \ge 0$.
The generators of the cohomology ring $H^*(M_C(2,\mc{L}),\mb{Q})$
as obtained by Newstead in \cite{new1} immediately give us the generators of the cohomology ring $H^*(\mc{G}(2,\mc{L})_\infty, \mb{Q})$.
However, not all elements in $H^i(\mc{G}(2, \mc{L})_\infty, \mb{Q})$
are monodromy invariant. As a result the specialization morphism, 
$\mr{sp}_i$ is neither injective nor surjective (see \eqref{ntor02}).
To compute explicitly the kernel and cokernel of $\mr{sp}_i$ we 
study the Gysin morphism from the intersection of the two 
components of $\mc{G}_{X_0}(2,\mc{L}_0)$ to the irreducible
components.

{\bf{Notation:}} Given any morphism $f:\mc{Y} \to S$ and a point $s \in S$, we denote by $\mc{Y}_s:=f^{-1}(s)$.
 The open unit disc is denoted by $\Delta$ and $\Delta^*:=\Delta\backslash \{0\}$ denotes the punctured disc.

 \emph{Acknowledgements} 
 We thank Prof. Nagaraj for helpful discussions.
 We are grateful to Prof. Newstead  for his interest in our work and feedback which 
 improved the exposition and made the article more succinct.
 The first named author wish to thank IMSc Chennai and 
 SRM University, AP for funding his research.
While writing this article, the second author was funded by CNPq Brazil and the third author was funded by a PNPD fellowship from CAPES Brazil. 
The second author is currently funded by  EPSRC grant number R/$162871-11-1$ and the third author is currently funded by SFB $326$-GAUS.

\section{Preliminaries}\label{seclim}
 In this section we recall the basic definitions related to limit mixed Hodge structures and the Gieseker moduli space of vector bundles with fixed determinant over an irreducible nodal curve.  
 For a comprehensive treatment of limit mixed Hodge structures we refer the reader to \cite{pet} and for the Gieseker moduli space \cite{tha}.
 
 \subsection{Preliminaries on limit mixed Hodge structures}
 We begin by recalling the definition of a mixed Hodge structure:
 \begin{defi}
  A \textit{mixed Hodge structure} is a triple $(H_{\mathbb{Z}},W_{\bullet},F^{\bullet})$ consisting of:
  \begin{enumerate}
   \item a lattice $H_{\mathbb{Z}}$,
   \item an increasing filtration $W_{\bullet}$ of $H_{\mathbb{Q}} = H_{\mathbb{Z}}\otimes_{\mathbb{Z}} \mathbb{Q}$,
   \item a decreasing filtration $F^{\bullet}$ of $H_{\mathbb{C}} = H_{\mathbb{Z}}\otimes_{\mathbb{Z}} \mathbb{C}$.   
  \end{enumerate}
such that the weight and Hodge filtrations are compatible in the sense that $F^{\bullet}$ induces a Hodge structure of weight $i$ on each of the \textit{graded pieces} 
$\mr{Gr}^W_{i} H_{\mb{Q}} := W_{i} H_{\mb{Q}} / W_{i-1} H_{\mb{Q}}$. 
 \end{defi}

 \begin{note}
 Let $\rho:\mc{Y} \to \Delta$ be a flat family of projective varieties, smooth over $\Delta^*$ and $\rho^{-1}(0)=\mc{Y}_{0}$ a simple normal crossings divisor in $\mc{Y}$. Denote by $\rho':\mc{Y}_{\Delta^*} \to \Delta^*$ the restriction of $\rho$ to $\Delta^*$.
 \end{note}

 By Ehresmann's theorem (see \cite[Theorem $9.3$]{v4}), for all $i \ge 0$, $\mb{H}_{\mc{Y}_{\Delta^*}}^i:=R^i{\rho}'_{*}\mb{Z}$
 is a local system over $\Delta^*$ with fiber $H^i(\mc{Y}_t,\mb{Z})$, for $t \in \Delta^*$.
 To these local systems, we associate the holomorphic vector bundles 
 $\mc{H}_{\mc{Y}_{\Delta^*}}^i:=\mb{H}_{\mc{Y}_{\Delta^*}}^i \otimes_{\mb{Z}} \mo_{\Delta^*}$
 called the \emph{Hodge bundle}.
 There exist holomorphic sub-bundles $F^p\mc{H}_{\mc{Y}_{\Delta^*}}^i \subset \mc{H}_{\mc{Y}_{\Delta^*}}^i$
such that for any $t \in \Delta^*$, the fibers $\left(F^p\mc{H}_{\mc{Y}_{\Delta^*}}^i\right)_t \subset \left(\mc{H}_{\mc{Y}_{\Delta^*}}^i\right)_t$
can be identified respectively with 
$F^pH^i(\mc{Y}_t,\mb{C}) \subset H^i(\mc{Y}_t,\mb{C})$ where $F^p$ denotes the Hodge filtration (see \cite[\S $10.2.1$]{v4}). By \cite[Definition $11.4$]{pet} there exists a canonical extension $\ov{\mc{H}}_{\mc{Y}}^i$ of ${\mc{H}}_{\mc{Y}_{\Delta^*}}^i$ to $\Delta$ such that $\ov{\mc{H}}_{\mc{Y}}^i$ is locally-free over $\Delta$.
 
\begin{defi}\label{defilmhs} 
 Denote by $\mf{h}$ the upper half plane $\{ z \in \mathbb{C} \ | \ \mr{Im}(z) > 0 \}$ and consider the universal covering of the punctured unit disc $\mf{h} \to \Delta^*$, given by $z \mapsto e^{2 \pi i z}$. 
 Let $e:\mf{h} \to \Delta^* \xrightarrow{j} \Delta$ denote the composed morphism where $j:\Delta^* \to \Delta$ is the inclusion morphism. Denote by $\mc{Y}_\infty:=\mc{Y} \times_{\Delta} \mf{h}$ the base change of the family $\mc{Y}$ over $\Delta$ to $\mf{h}$, by the morphism $e$. We now define a mixed Hodge structure on $H^i(\mc{Y}_{\infty},\mb{Z})$:

 The Hodge sub-bundle $F^p\mc{H}_{\mc{Y}_{\Delta^*}}$ on $\Delta^*$ extends to $\Delta$ by setting 
 $F^p\ov{\mc{H}}_{\mc{Y}}^i:= j_*\left(F^p\mc{H}_{\mc{Y}_{\Delta^*}}^i\right) \cap  \ov{\mc{H}}_{\mc{Y}}^i$.
  
   
   By \cite[XI-$8$]{pet}, for a choice of the parameter $t$ on $\Delta$, the central fiber of the canonical extension $\ov{\mc{H}}_{\mc{Y}}^i$ can be identified with the cohomology group $H^i(\mc{Y}_{\infty},\mb{C})$:  
 \begin{equation}\label{tor23}
  g^i_{_t}:H^i(\mc{Y}_{\infty},\mb{C}) \xrightarrow{\sim} \left(\ov{\mc{H}}_{\mc{Y}}^i\right)_0.
 \end{equation}
  As a consequence, there exist Hodge filtrations on $H^i(\mc{Y}_{\infty},\mb{C})$ defined by
\[F^pH^i(\mc{Y}_{\infty},\mb{C}):=(g_{_t}^i)^{-1}\left(F^p\ov{\mc{H}}_{\mc{Y}}^i\right)_0.\]

To define a weight filtration on $H^i(\mc{Y}_{\infty},\mb{Z})$, we use the local monodromy transformations. For any $s \in \Delta^*$ and $i \ge 0$, denote by 
\[T_{s,i}: H^i(\mc{Y}_s,\mb{Z}) \to H^i(\mc{Y}_s,\mb{Z}) \, \mbox{ and }\, T_{s,i}^{\mb{Q}}: H^i(\mc{Y}_s,\mb{Q}) \to H^i(\mc{Y}_s,\mb{Q})\]
the \emph{local monodromy transformations} associated to the local system $\mb{H}_{\mc{Y}_{\Delta^*}}^i$
   defined by parallel transport along a counterclockwise loop about $0 \in \Delta$ (see \cite[\S $11.1.1$]{pet}).
 By \cite[Theorem II.$1.17$]{deli2}  (see also \cite[Proposition I.$7.8.1$]{kuli}) the automorphism extends to a $\mb{Q}$-automorphism 
\begin{equation}\label{int01}
 T_i: H^i(\mc{Y}_{\infty},\mb{Q}) \to H^i(\mc{Y}_{\infty},\mb{Q}).
\end{equation}

  Let $N_i$ be the logarithm of the monodromy operator $T_i$.
  By \cite[Lemma-Definition $11.9$]{pet}, there exists an unique increasing \emph{monodromy weight filtration} $W_\bullet$ on $H^i(\mc{Y}_\infty,\mb{Q})$ such that:
 \begin{enumerate}
  \item  for $j \ge 2$, $N_i(W_jH^i(\mc{Y}_\infty,\mb{Q})) \subset W_{j-2}H^i(\mc{Y}_\infty,\mb{Q})$ and
  \item the map $N_i^l: \mr{Gr}^W_{i+l} H^i(\mc{Y}_\infty,\mb{Q}) \to \mr{Gr}^W_{i-l} H^i(\mc{Y}_\infty,\mb{Q})(-l)$ 
  is an isomorphism for all $l \ge 0$.
   \end{enumerate}
  By \cite[Theorem $6.16$]{schvar} the induced filtrations on  $H^i(\mc{Y}_{\infty},\mb{C})$ define a mixed Hodge structure $(H^i(\mc{Y}_{\infty},\mb{Z}),W_\bullet,F^\bullet)$, called the \emph{limit mixed Hodge structure} on $H^i(\mc{Y}_{\infty},\mb{Z})$.
  \end{defi}
  
   The main purpose for using the limit mixed Hodge structure on $H^i(\mc{Y}_{\infty},\mb{Z})$ is that when equipped with this Hodge structure, the \emph{specialization morphism}
  \[\mr{sp}_i: H^i(\mc{Y}_0, \mb{Q}) \to H^i(\mc{Y}_\infty, \mb{Q})\]
  is indeed a morphism of mixed Hodge structures.

In the case when the central fiber $\mc{Y}_0$ is a reduced simple normal crossings divisor (as will be the case for us), the computation for the weight filtration is vastly simplified by using the Steenbrink spectral sequence. We now recall the limit weight spectral sequence in the case when the special fiber consists of exactly \emph{two} irreducible components. We will apply this result to our situation in \S \ref{sec:gieslmhs}.

  \begin{prop}\label{tor26}
  Suppose the special fiber $\mc{Y}_0$ consists of exactly \emph{two} irreducible components, say $Y_1$ and $Y_2$.  
 The \emph{limit weight spectral sequence} $^{^{\infty}}_{_W}E_1^{p,q} \Rightarrow H^{p+q}(\mc{Y}_\infty , \mb{Q})$ consists of the following terms:
 \begin{enumerate}
  \item if $|p| \ge 2$, then $^{^{\infty}}_{_W}E_1^{p,q}=0$,
  \item $^{^{\infty}}_{_W}E_1^{1,q}=H^q(Y_1 \cap Y_2,\mb{Q})(0)$, $^{^{\infty}}_{_W}E_1^{0,q}=H^q(Y_1,\mb{Q})(0) \oplus H^q(Y_2,\mb{Q})(0)$  and $^{^{\infty}}_{_W}E_1^{-1,q}=H^{q-2}(Y_1 \cap Y_2,\mb{Q})(-1)$,
  \item the differential map $d_1:\, ^{^{\infty}}_{_W}E_1^{0,q} \to \, ^{^{\infty}}_{_W}E_1^{1,q}$ is the restriction morphism and \[d_1:\, ^{^{\infty}}_{_W}E_1^{-1,q} \to \, ^{^{\infty}}_{_W}E_1^{0,q}\] is the Gysin morphism.
 \end{enumerate}
The limit weight spectral sequence $^{^{\infty}}_{_W}E_1^{p,q}$ degenerates at $E_2$.
Similarly, the \emph{weight spectral sequence } $_{_W}E_1^{p,q} \Rightarrow H^{p+q}(\mc{Y}_0 , \mb{Q})$  on $\mc{Y}_0$ consists of the following terms:
 \begin{enumerate}
  \item for $p \ge 2$ or $p<0$, we have $_{_W}E_1^{p,q}=0$,
  \item $_{_W}E_1^{1,q}=H^q(Y_1 \cap Y_2,\mb{Q})(0)$ and $_{_W}E_1^{0,q}=H^q(Y_1,\mb{Q})(0) \oplus H^q(Y_2,\mb{Q})(0)$,
  \item the differential map $d_1:\, _{_W}E_1^{0,q} \to \, _{_W}E_1^{1,q}$ is the restriction morphism.
 \end{enumerate}
 The spectral sequence $_{_W}E_1^{p,q}$ degenerates at $E_2$. 
 \end{prop}
 
 \begin{proof}
  See \cite[Corollary $11.23$]{pet} and \cite[Example $3.5$]{ste1}.
 \end{proof}

 We now recall without proof some results that we use repeatedly in \S \ref{sec:proofofthm}.
 
 \begin{rem}\label{ntor03}
 By \cite[Corollary $2.4$]{mumf} we have the following exact sequence of mixed Hodge structures:
  \begin{equation}\label{ntor02}
 H^{i-2}(Y_1 \cap Y_2, \mb{Q})(-1) \xrightarrow{f_i} H^i(\mc{Y}_0,\mb{Q}) \xrightarrow{\mr{sp}_i} H^i(\mc{Y}_\infty, \mb{Q}) \xrightarrow{g_i} \mr{Gr}^W_{i+1} H^i(\mc{Y}_\infty, \mb{Q}) \to 0,
  \end{equation}
where $f_i$ is the natural morphism induced by the Gysin morphism from $H^{i-2}(Y_1 \cap Y_2, \mb{Q})(-1)$ to $H^i(Y_1,\mb{Q}) \oplus H^i(Y_2, \mb{Q})$
(use the Mayer-Vietoris sequence associated to $Y_1 \cup Y_2$)
and $g_i$ is the natural projection.
\end{rem}

 \subsection{Relative Gieseker moduli space}\label{sec:gies}
 
 \begin{note}\label{tor33}
 Let $X_0$ be an irreducible nodal curve of genus $g \ge 2$, with exactly
 one node, say $x_0$. Denote by $\pi_0:\widetilde{X}_0 \to X_0$ the normalization map.
 Since the moduli space of stable curves 
 is complete, there exists a regular, flat family of projective curves
  \[\pi_1:\mc{X} \to \Delta\] smooth over 
   $\Delta^*$ and central fiber isomorphic to $X_0$ (see \cite[Theorem B$.2$]{bake}).
   Fix an invertible sheaf $\mc{L}$ on $\mc{X}$ of relative odd degree, say $d$. 
  Set $\mc{L}_0:=\mc{L}|_{X_0}$, the restriction of $\mc{L}$ to the central fiber.
  Denote by $\widetilde{\mc{L}}_0:=\pi_0^*(\mc{L}_0)$.
 \end{note}
 
  By \cite[\S $6$]{tha} there exists a regular, flat, projective family 
    \[\pi_2:\mc{G}(2,\mc{L}) \to \Delta\]
    called the \emph{relative Gieseker moduli spaces of rank $2$ semi-stable sheaves on $\mc{X}$ with determinant 
  $\mc{L}$} such that for all $s \in \Delta^*$, $\mc{G}(2,\mc{L})_s:=\pi_2^{-1}(s)=M_{\mc{X}_s}(2,\mc{L}_s)$ is the moduli space of rank $2$
  stable vector bundles on $\mc{X}_s$ with determinant $\mc{L}_s$ and
  the central fiber $\pi_2^{-1}(0)$, denoted 
  $\mc{G}_{X_0}(2,\mc{L}_0)$, is a reduced simple normal crossings divisor of $\mc{G}(2,\mc{L})$. In particular, $\mc{G}(2,\mc{L})$ is smooth over $\Delta^*$.
  
   Denote by $M_{\widetilde{X}_0}(2,\widetilde{\mc{L}}_0)$ the fine moduli space of semi-stable sheaves of rank $2$ and  
  determinant $\widetilde{\mc{L}}_0$ over $\widetilde{X}_0$.
  By \cite[\S $6$]{tha}, $\mc{G}_{X_0}(2,\mc{L}_0)$ can be written as the union of two irreducible components, 
  say $\mc{G}_0$ and $\mc{G}_1$, where 
  $\mc{G}_1$ (resp. $\mc{G}_0 \cap \mc{G}_1$) is isomorphic to a $\p3$ (resp. $\mb{P}^1 \times \mb{P}^1$)-bundle 
  over $M_{\widetilde{X}_0}(2,\widetilde{\mc{L}}_0)$.  
  Denote by \[\rho_1: \mc{G}_0 \cap \mc{G}_1 \longrightarrow M_{\widetilde{X}_0}(2,\widetilde{\mc{L}}_0)\, 
\mbox{ and } \rho_2: \mc{G}_1 \longrightarrow M_{\widetilde{X}_0}(2,\widetilde{\mc{L}}_0)\]
the natural projections.

  \subsection{Limit mixed Hodge structure for the relative Gieseker moduli space}\label{sec:gieslmhs}
  Let $\mc{G}(2,\mc{L})_\infty$ denote the base change of the family $\mc{G}(2,\mc{L})$ over $\Delta$ to the universal cover $\mf{h}$ by the morphism  $e:\mf{h} \to \Delta^* \xrightarrow{j} \Delta$. Using Definition \ref{defilmhs}, we can 
 equip $H^{i}(\mc{G}(2,\mc{L})_\infty, \mb{Q})$ with a limit mixed Hodge structure, for any $i$ and a specialization morphism as:
  \[\mr{sp}_i: H^i(\mc{G}_{X_0}(2,\mc{L}_0), \mb{Q}) \to H^i(\mc{G}(2,\mc{L})_\infty, \mb{Q}).\]

 Denote by $\iota_1:\mc{G}_0 \cap \mc{G}_1 \hookrightarrow \mc{G}_0$ and $\iota_2:\mc{G}_0 \cap \mc{G}_1 \hookrightarrow \mc{G}_1$
  the natural inclusions. Using the Mayer-Vietoris sequence, we have the exact sequence:
  \begin{equation}\label{mv01}
 H^{i-1}(\mc{G}_0 \cap \mc{G}_1, \mb{Q}) \to  H^i(\mc{G}_{X_0}(2,\mc{L}_0), \mb{Q}) \to H^i(\mc{G}_0,\mb{Q}) \oplus H^i(\mc{G}_1, \mb{Q}) \xrightarrow{\iota_1^*-\iota_2^*} H^i(\mc{G}_0 \cap \mc{G}_1, \mb{Q}).
  \end{equation}
  By \cite[Proposition B.$30$]{pet}, one can check that the composed morphism 
  \begin{equation}\label{mv02}
   H^{i-2}(\mc{G}_0 \cap \mc{G}_1,\mb{Q})(-1) \xrightarrow{(\iota_{1,_*},\iota_{2,_*})} H^i(\mc{G}_0,\mb{Q}) \oplus H^i(\mc{G}_1, \mb{Q}) \xrightarrow{(\iota_1^*-\iota_2^*)} H^i(\mc{G}_0 \cap \mc{G}_1, \mb{Q})
  \end{equation}
  is a zero map. As $\mc{G}_0 \cap \mc{G}_1, \mc{G}_0$ and $\mc{G}_1$ are smooth, $H^{i-2}(\mc{G}_0 \cap \mc{G}_1, \mb{Q})(-1), 
  H^i(\mc{G}_0, \mb{Q})$ and $H^i(\mc{G}_1, \mb{Q})$ are pure Hodge structures of weight $i$.
  Therefore, the second morphism in \eqref{mv01} induces an injective morphism 
  \[\mr{Gr}^W_i H^i(\mc{G}_{X_0}(2,\mc{L}_0), \mb{Q}) \hookrightarrow H^i(\mc{G}_0, \mb{Q}) \oplus H^i(\mc{G}_1, \mb{Q}).\]
  As a result, the Gysin morphism \[(\iota_{1,_*}, \iota_{2,_*}): H^{i-2}(\mc{G}_0 \cap \mc{G}_1, \mb{Q})(-1) \to H^i(\mc{G}_0, \mb{Q}) \oplus H^i(\mc{G}_1, \mb{Q})\]
  factors through $\mr{Gr}^W_i H^i(\mc{G}_{X_0}(2,\mc{L}_0), \mb{Q})$. Denote by 
  \[f_i: H^{i-2}(\mc{G}_0 \cap \mc{G}_1, \mb{Q})(-1) \to \mr{Gr}^W_i H^i(\mc{G}_{X_0}(2,\mc{L}_0), \mb{Q}) \hookrightarrow H^i(\mc{G}_{X_0}(2,\mc{L}_0), \mb{Q}),\]
  the composed morphism. Using \cite[Theorem $4.2$]{mumf}, we observe that the kernel of the morphism $f_i$ is isomorphic
  to $H^{i-4}(M_{\widetilde{X}_0}(2,\widetilde{\mc{L}}_0), \mb{Q})(-2)$.  
  Applying Remark \ref{ntor03}, we then get the following exact sequence of mixed Hodge structures:
    \begin{eqnarray}\label{ntor09}
 &0 \to H^{i-4}(M_{\widetilde{X}_0}(2,\widetilde{\mc{L}}_0), \mb{Q})(-2) \to H^{i-2}(\mc{G}_0 \cap \mc{G}_1, \mb{Q})(-1) \xrightarrow{f_i} H^i(\mc{G}_{X_0}(2,\mc{L}_0),\mb{Q}) \xrightarrow{\mr{sp}_i} \nonumber\\
 &\xrightarrow{\mr{sp}_i} H^i(\mc{G}(2,\mc{L})_\infty, \mb{Q}) \xrightarrow{g_i} \mr{Gr}^W_{i+1} H^i(\mc{G}(2,\mc{L})_\infty, \mb{Q}) \to 0,
  \end{eqnarray}
  where $g_i$ is the natural projection.

 \begin{rem}\label{isom2} 
   Applying Proposition \ref{tor26} to the family $\pi_2$, we have
 \[\mr{Gr}^W_4H^3(\mc{G}(2,\mc{L})_\infty, \mb{Q})=\, ^{^{\infty}}_{_W}E_2^{-1,4}=\ker(\, ^{^{\infty}}_{_W}E_1^{-1,4} \to \, ^{^{\infty}}_{_W}E_1^{0,4}) =\]
 \[= \ker(H^2(\mc{G}_0 \cap \mc{G}_1,\mb{Q})(-1) \xrightarrow{(\iota_{1_*},\iota_{2_*})} H^4(\mc{G}_0,\mb{Q}) \oplus H^4(\mc{G}_1,\mb{Q})) \cong 
 H^0(M_{\widetilde{X}_0}(2,\widetilde{\mc{L}}_0), \mb{Q})(-2)= \mb{Q},\]
 where the second last isomorphism follows from \cite[Theorem $4.2$]{mumf}.
  Moreover, by the definition of limit mixed Hodge structure as given in Definition \ref{defilmhs}, this implies 
  \[\mr{Gr}_2^WH^3(\mc{G}(2,\mc{L})_{\infty}, \mb{Q}) = \mr{Gr}^W_4H^3(\mc{G}(2,\mc{L})_\infty,\mb{Q}) \cong \mb{Q}.\]
 \end{rem}
  
\section{Generators of the cohomology ring of the Gieseker moduli space}\label{sec:proofofthm}
 
In this section, we give the generators of the Gieseker's moduli space 
of rank $2$, semi-stable sheaves with fixed determinant over an irreducible nodal curve with one node, thereby generalizing the 
classical result of Newstead \cite[Theorem $1$]{new1} to this case. 
 
\begin{note}\label{note:ntor11}
 Recall, \cite[Theorem $1$]{new1} states that there exists 
 \[\alpha_0 \in H^2(M_{\widetilde{X}_0}(2,\widetilde{\mc{L}}_0),\mb{Z}), \beta_0 \in H^4(M_{\widetilde{X}_0}(2,\widetilde{\mc{L}}_0),\mb{Z})\]
 and $\psi'_1, \psi'_2, ...,\psi'_{2g-2} \in H^3(M_{\widetilde{X}_0}(2,\widetilde{\mc{L}}_0),\mb{Z})$, generating 
  $H^*(M_{\widetilde{X}_0}(2,\widetilde{\mc{L}}_0),\mb{Q})$. 
 \end{note}
 
\begin{proof}[Proof of Theorem \ref{th:gen05}]  
For any $s \in \Delta^*$, 
  choose a symplectic basis $e_1,...,e_{2g}$ of $H^1(\mc{X}_s,\mb{Z})$ such that $e_i \cup e_{i+g}=-[\mc{X}_s]^\vee$ and $e_i \cup e_j=0$ for $|j-i| \not= g$,
 where $[\mc{X}_s]^\vee$ denotes the (Poincar\'{e}) dual fundamental class of $\mc{X}_s$.
 By \cite[Proposition $1$]{mumn}, there exists a unimodular isomorphism of pure Hodge structures:
 \[\phi_s:H^1(\mc{X}_s,\mb{Z}) \xrightarrow{\sim} H^3(M_{\mc{X}_s}(2,\mc{L}_s),\mb{Z}).\]
 Denote by $\psi_i:=\phi_s(e_i)$ for $1 \le i \le 2g$. By \cite[Theorem $1$]{new1}, there exists $\alpha_s \in H^2(M_{\mc{X}_s}(2,\mc{L}_s), \mb{Z})$
 and $\beta_s \in H^4(M_{\mc{X}_s}(2,\mc{L}_s), \mb{Z})$ such that the cohomology ring $H^*(M_{\mc{X}_s}(2,\mc{L}_s), \mb{Q})$ is generated by $\alpha_s, \beta_s$
 and $\psi_i$ for $1 \le i \le 2g$. Using Ehresmann's theorem, we have isomorphisms:
 \[\tau_s^i: H^i(M_{\mc{X}_s}(2,\mc{L}_s),\mb{Z}) \to H^i(\mc{G}(2,\mc{L})_\infty,\mb{Z})\]
 commuting with the cup-product (cup-product commutes with pull-back by continuous maps)
\[ \mbox{i.e., }\, \, \tau_s^i(\sigma_1) \cup \tau_s^j(\sigma_2)=\tau_s^{i+j}(\sigma_1 \cup \sigma_2)
 \mbox{ for } \sigma_1 \in H^i(M_{\mc{X}_s}(2,\mc{L}_s),\mb{Z}) \mbox{ and }
 \sigma_2 \in H^j(M_{\mc{X}_s}(2,\mc{L}_s),\mb{Z}).\]
   Denote by $\alpha_\infty:=\tau^2_s(\alpha_s)$, $\beta_\infty:=\tau^4_s(\beta_s)$ and $\psi_i^\infty:=\tau^3_s(\psi_i)$
  for $1 \le i \le 2g$. Since the isomorphism $\tau_s^i$ commutes with cup-product, 
 the cohomology ring $H^*(\mc{G}(2,\mc{L})_\infty,\mb{Q})$ is generated by $\alpha_\infty, \beta_\infty$ and $\psi_i^{\infty}$ for $1 \le i \le 2g$.
 
 By Remark \ref{isom2} we have $\mr{Gr}^W_4H^3(\mc{G}(2,\mc{L})_\infty,\mb{Q}) \cong \mr{Gr}^W_2H^3(\mc{G}(2,\mc{L})_\infty,\mb{Q}) \cong \mb{Q}$. 
 After changing basis if necessary, we can assume that $\psi_{2g}^\infty$ (resp. $\psi_{g}^\infty$)  
  generates $\mr{Gr}^W_4H^3(\mc{G}(2,\mc{L})_\infty,\mb{Q})$ (resp. $\mr{Gr}^W_2H^3(\mc{G}(2,\mc{L})_\infty,\mb{Q})$) and the image of $\psi_i^\infty$ in $\mr{Gr}^W_4H^3(\mc{G}(2,\mc{L})_\infty,\mb{Q})$ for any $1 \le i \le 2g-1$, is $0$.
   Let $\xi_0$ be a generator of $H^2(\mb{P}^1, \mb{Q})$, $\pr_i$ are the natural 
  projections from $\mb{P}^1 \times \mb{P}^1$ to $\mb{P}^1$ and $\xi_i:=\pr_i^*(\xi_0)$ for $i=1,2$.
 Recall, by the Leray-Hirsch theorem (see \cite[Theorem $7.33$]{v4}), we have:
 \begin{equation}\label{ntor04}
  H^i(\mc{G}_0 \cap \mc{G}_1) \cong H^i(M_{\widetilde{X}_0}(2,\widetilde{\mc{L}}_0)) \oplus H^{i-2}(M_{\widetilde{X}_0}(2,\widetilde{\mc{L}}_0)) \otimes (\mb{Q}\xi_1 \oplus \mb{Q}\xi_2) \oplus H^{i-4}(M_{\widetilde{X}_0}(2,\widetilde{\mc{L}}_0)) \xi_1\xi_2.
  \end{equation}
  Since $\mc{G}_0 \cap \mc{G}_1$ is smooth, rationally connected, by \cite[Corollary $4.18$]{Deba} we have $H^0(\Omega^1_{\mc{G}_0 \cap \mc{G}_1})=0$.
  Note that, $H^0(\Omega^1_{\mc{G}_0 \cap \mc{G}_1})$ is isomorphic to $H^{1,0}(\mc{G}_0 \cap \mc{G}_1, \mb{C})$, under the Hodge decomposition of 
  $H^1(\mc{G}_0 \cap \mc{G}_1, \mb{C})$ (see \cite[Corollary $7.4$]{v4}). Hence, \[H^1(\mc{G}_0 \cap \mc{G}_1, \mb{C}) \cong 
                            H^{1,0}(\mc{G}_0 \cap \mc{G}_1, \mb{C}) \oplus  \overline{H^{1,0}(\mc{G}_0 \cap \mc{G}_1, \mb{C})}=0,\]
where $\overline{(\,)}$ denotes complex conjugation.
  Therefore the morphism  $\mr{sp}_3$ in \eqref{ntor09} is injective. Moreover \eqref{ntor09} also implies that $f_2$ is injective and non-zero.
  Let $\xi_1^{(2)}, \xi_2^{(2)} \in H^2(\mc{G}_{X_0}(2,\mc{L}_0),\mb{Q})$ such that $\xi_1^{(2)}=f_2(1)$ for $1 \in H^0(\mc{G}_1 \cap \mc{G}_2,\mb{Q})$ and
 $\mr{sp}_2(\xi_2^{(2)})=\alpha_\infty$. 
 Denote by \[\xi_i^{(3)}:=\mr{sp}_3^{-1}(\psi_i^\infty) \in H^3(\mc{G}_{X_0}(2,\mc{L}_0),\mb{Q}),\, \mbox{ for }\, 1 \le i \le 2g-1.\] 
 Let $\xi_j^{(4)}$ for $1 \le j \le 3$ such that $\mr{sp}_4(\xi_1^{(4)})=\beta_\infty$ and under the identification \eqref{ntor04}, 
 \[f_4(\xi_1 \oplus \xi_2)=\xi_2^{(4)} \mbox{  and } f_4(\alpha_0)=\xi_3^{(4)}.\] 
  Let $\xi_i^{(5)}=f_5(\psi_i')$ for $1 \le i \le 2g-2$ 
 and $\xi_1^{(6)}:=f_6(\beta_0)$.
 Let $\xi_2^{(6)} \in H^6(\mc{G}_{X_0}(2,\mc{L}_0),\mb{Q})$ such that $\mr{sp}_6(\xi_2^{(6)})=\psi_g^\infty\psi_{2g}^\infty$.

Let $V \subset H^*(\mc{G}_{X_0}(2,\mc{L}_0),\mb{Q})$ be the sub-ring of the cohomology ring, generated by $\xi_j^{(i)}$ for $1 \le i \le 6$. 
Note that, $V$ is 
equipped with a natural grading coming from the cohomology ring. Denote by $V_k := V \cap H^k(\mc{G}_{X_0}(2,\mc{L}_0), \mb{Q})$.
To prove that $V=H^*(\mc{G}_{X_0}(2,\mc{L}_0), \mb{Q})$, we need to prove that for each $k$, $\mr{sp}_k(V_k)=\mr{Im}(\mr{sp}_k)$ and $\Ima(f_k) \subset V_k$.
By \cite[Theorem $1$]{new1}, the identification \eqref{ntor04} implies that $H^*(\mc{G}_0 \cap \mc{G}_1,\mb{Q})$
is generated by $\alpha_0, \beta_0, \xi_1, \xi_2$ and $\psi'_i$ for $1 \le i \le 2g-2$.
Using \eqref{ntor09} we conclude that $\Ima(f_k) \subset V_k$.
We now show that $\mr{sp}_k(V_k)=\mr{Im}(\mr{sp}_k)$. 

Using \cite[Remark $5.3$]{kingn} one can check that 
$H^k(\mc{G}(2,\mc{L})_\infty,\mb{Q})$ is $\mb{Q}$-generated by monomials of 
the form $\alpha_\infty^{i}\beta_{\infty}^j\psi_{i_1}^\infty...\psi_{i_t}^\infty$, where $i+t<g, j+t<g, i_1<i_2<...<i_t$ and $2i+4j+3t=k$. Since cup-product 
is a morphism of mixed Hodge structures, we conclude that $W_kH^k(\mc{G}(2,\mc{L})_\infty,\mb{Q})$ is generated by all such monomials 
for which either $i_t<2g$ or $\{g,2g\} \in \{i_1, i_2,..., i_t\}$. Hence, $\mr{sp}_k(V_k)=\Ima(\mr{sp}_k)$. The minimality of the set 
$\{\xi_i^{(j)}\}_{1 \le j \le 6}$ of generators 
is straightforward. This proves the theorem.
\end{proof}

\begin{rem}\label{rem:simp}
We remark here that in \cite[Theorem $1.1$]{mumf}, Dan and Kaur show that 
the generators of the cohomology ring of $U_{X_0}(2,\mc{L}_0)$ are
given as follows:
 there exist
  $\alpha_{2j} \in H^{2j}(U_{X_0}(2,\mc{L}_0), \mb{Q})$ for $1 \le j \le 3$ and
  $\gamma_i \in H^3(U_{X_0}(2,\mc{L}_0), \mb{Q})\mbox{ for } 1 \le i \le 2g-1$
  forming a minimal set of generators of the cohomology ring  $H^*(U_{X_0}(2,\mc{L}_0), \mb{Q})$. 
 \end{rem}

 \bibliographystyle{plain}

\end{document}